\documentclass[12pt]{notices}
\usepackage{amsmath, amsthm, amssymb, hyperref, color}
\hypersetup{
     colorlinks = true,
     linkcolor = black,
     anchorcolor = black,
     citecolor = black,
     filecolor = black,
     urlcolor = black
     }
\usepackage{graphicx}
\usepackage{caption}
\usepackage[all]{xypic}
\usepackage{verbatim}
\usepackage{chemfig, chemnum}
\usepackage{tikz}
\usepackage[linesnumbered,lined,commentsnumbered]{algorithm2e}
\usepackage{caption}
\usepackage{subcaption}
\oddsidemargin \evensidemargin
\usepackage{makecell}
\usepackage{array}
\newcolumntype{?}{!{\vrule width 1pt}}

\newtheorem{theorem}{Theorem}
\newtheorem{proposition}[theorem]{Proposition}

\newtheorem{definition}[theorem]{Definition}

\newtheorem{remark}[theorem]{Remark}

\theoremstyle{definition}

\newcommand{\PP}{\mathbb{P}}
\newcommand{\RR}{\mathbb{R}}
\newcommand{\QQ}{\mathbb{Q}}
\newcommand{\CC}{\mathbb{C}}

\title{\textbf{3264 Conics in a Second}}

\author{\bf Paul Breiding, Bernd Sturmfels and Sascha Timme}
\date{}

\begin{document}
\maketitle

This article and its accompanying web interface present \emph{Steiner's conic problem} and a discussion on how \emph{enumerative} and \emph{numerical} algebraic geometry complement each other.
The intended audience is students at an advanced undergrad level.
Our readers can see
current computational tools in action on a
geometry problem that has inspired scholars for two centuries.
The take-home message is that numerical methods in algebraic geometry
are fast and reliable.

We begin by recalling the statement of Steiner's conic problem.
A {\em conic} in the plane $\RR^2$ is the set of solutions to a quadratic equation
$A(x,y) = 0$, where
\begin{equation}
\label{eq:conicA}
A(x,y)  =  a_1 x^2 + a_2 x y + a_3 y^2 + a_4 x + a_5 y + a_6 .
\end{equation}
If there is a second conic
\begin{equation}
\label{eq:conicU}
 U(x,y) =  u_1 x^2 +u_2 x y +u_3 y^2 + u_4 x + u_5 y + u_6,
\end{equation}
then the two conics intersect in four points in $\CC^2$, counting multiplicities and counting intersections at points at infinity,
provided $A$ and $U$ are irreducible and not multiples of each other.
This is the content of {\em B\'ezout's Theorem}. To take into account the points of intersection at infinity, algebraic geometers like to replace the
affine plane $\CC^2$ with the complex projective plane $\PP^2_\CC$. In the following, when we write 'count', we always mean counting solutions in projective space. Nevertheless, for our exposition we work with $\CC^2$.

A solution $(x,y)$ of the system $A=U=0$ has multiplicity $\geq 2$
if it is a zero of the {\em Jacobian determinant}
\begin{equation}
\label{eq:jacobian}
\begin{array}{rl}
   &\frac{\partial A}{ \partial x} \cdot \frac{\partial U}{\partial y} -
  \frac{\partial A}{ \partial y} \cdot \frac{\partial U}{\partial  x}\,=\,
     2 (a_1 u_2 - a_2 u_1) x^2  \qquad \\
    & \,+4 (a_1 u_3 - a_3 u_1) x y
   + \cdots + (a_4 u_5 - a_5 u_4).
 \end{array}
\end{equation}
Geometrically,
the conic $U$ is {\em tangent} to the conic $A$ if
(\ref{eq:conicA}), (\ref{eq:conicU}) and (\ref{eq:jacobian}) are zero
 for some $(x,y)\in \CC^2$.
For instance, Figure \ref{great_figure} shows a red ellipse and five other blue conics, which are tangent to the red ellipse. {\em Steiner's conic problem} asks  the following question:

\begin{figure*}[t]
  \begin{center}
\includegraphics[width = 0.5\textwidth]{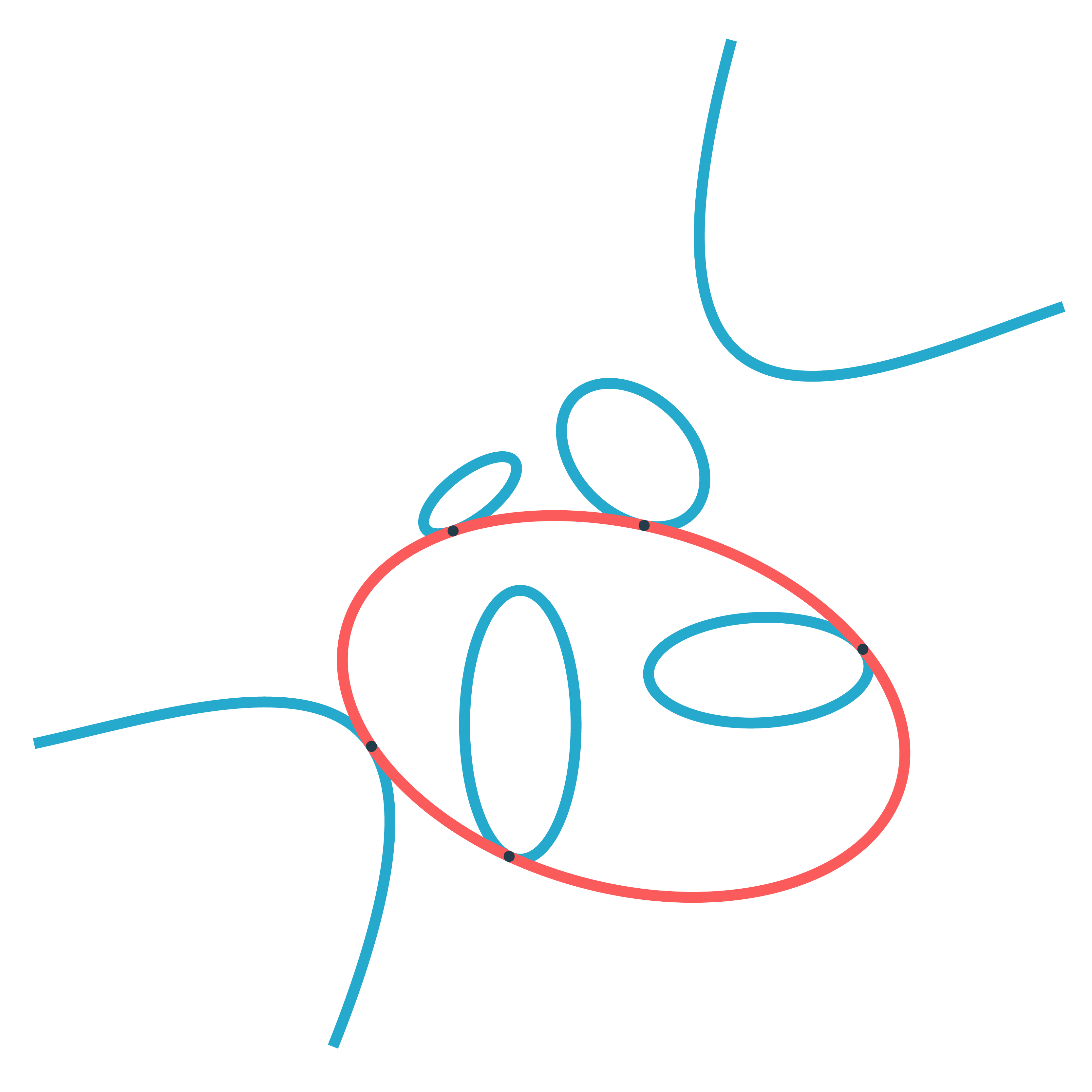}
\vspace{-0.31in}
\end{center}
  \caption{The red ellipse is tangent to four blue ellipses and one blue hyperbola.\label{great_figure}}
  \end{figure*}

  \begin{quote}
  \emph{How many conics in the plane are tangent to five given conics in general position?}
  \end{quote}

\noindent The number is five because
each tangency condition removes one
of the five degrees of freedom in a conic.

The present article concerns the following two subject areas and how they approach Steiner's problem:\\

\begin{tabular}{l}
 {\bf Enumerative algebraic geometry:} \\
\hspace{0.4em}{\em How many conics are tangent to five conics?} \\
{\bf Numerical algebraic geometry:} \\
\hspace{0.4em}{\em How do we find all conics tangent to five conics?}\\
\end{tabular}
\vspace{0.4em}

\noindent The first question is the original conic problem, first asked in 1848 by Steiner who suggested the answer $7776$.
That  number turned out to be incorrect. In the year 1864 Chasles gave the correct answer of $3264$.
This was further developed by Schubert, whose~1879 book
led to Hilbert's 15th problem, and thus to
 the 20th century development of enumerative algebraic geometry.
The number $3264$ appears prominently in the title of
the textbook by Eisenbud and Harris \cite{EH}.
A delightful introduction to Steiner's problem was presented by
Bashelor, Ksir and Traves~in~\cite{BKT}.

Numerical algebraic geometry is a younger subject. It started about 40 years ago, going back at least to \cite{GZ}.
The textbook by Sommese and Wampler~\cite{SW}
 is a standard reference.  It focuses  on numerical solutions to polynomial equations.
 The field is now often seen as a branch of applied mathematics. But, as we demonstrate in this article,
its methodology can be used in pure mathematics~too.

An instance of our problem is given by a list of $30 = 5 \times 6 $ coefficients in $\RR$ or $\CC$:
\begin{small}
\begin{equation}
\label{eq:instance}
  \begin{array}{l}
   A(x,y) \,=\, a_1 x^2 + a_2 xy + a_3 y^2 + a_4 x + a_5 y + a_6 , \\[0.3em]
   B(x,y) \,=\, b_1 x^2 + b_2 xy + b_3 y^2 + b_4 x + b_5 y + b_6 , \\[0.3em]
   C(x,y) \,=\, c_1 x^2 + c_2 xy + c_3 y^2 + c_4 x + c_5 y + c_6 , \\[0.3em]
   D(x,y) \,=\, d_1 x^2 + d_2 xy + d_3 y^2 + d_4 x + d_5 y + d_6 , \\[0.3em]
   E(x,y) \,=\, e_1 x^2 + e_2 xy + e_3 y^2 + e_4 x + e_5 y + e_6 .
  \end{array}
\end{equation}
\end{small}

\noindent By eliminating the two unknowns $x$ and $y$ from the three equations (\ref{eq:conicA}), (\ref{eq:conicU}) and (\ref{eq:jacobian}),
we can write the tangency condition directly in terms of the  $12=6+6$ coefficients
$a_1,\ldots,a_6,u_1, \ldots,u_6$ of $A$ and $U$:
\begin{equation}
\label{eq:tact}
 \begin{array}{rl}
  \mathcal{T}(A,U) \,=\,& \mkern-12mu
 256  a_1^4  a_3^2  u_3^2  u_6^4 - 128  a_1^4  a_3^2  u_3  u_5^2  u_6^3
 \\
  &\mkern-12mu + \, 16  a_1^4  a_3^2  u_5^4  u_6^2 +
\cdots + a_5^4 a_6^2 u_1^2 u_2^4.
\end{array}
\end{equation}
The polynomial $\mathcal{T}$ is a sum of $3210$ terms. It is of degree six in the variables $a_1,\ldots,a_6$ and of degree six in $u_1,\ldots,u_6$.
Known classically as the {\em tact invariant},
it vanishes precisely when the two conics are tangent.

If the coefficients are general, we can assume
that each conic $U$ that is tangent to $A,B,C,D,E$ has
nonzero constant term $u_6$. We can then
set $u_6~=~1$.
Steiner's problem for the conics $A,B,C,D,E$ now translates into a system
of five polynomial equations in five unknowns $u_1,u_2,u_3,u_4,u_5$.
Each of the five tangency constraints is an equation of degree six:
\begin{equation}
    \mathcal{T}(A,U) \,=\, \mathcal{T}(B,U) \,=\, \cdots
    \, = \, \mathcal{T}(E,U) \,=\, 0.
  \label{eq:formulation1}
\end{equation}
Steiner used B\'ezout's Theorem to argue that these equations
have $6^5 = 7776$ solutions. However, this number overcounts because
there is a {\em Veronese surface} of extraneous solutions $U$, namely
the squares of linear forms. These degenerate conics have the form
 $$ U(x,y) \,\,=\,\, (x,y,1) \cdot \ell^T\ell \cdot (x,y,1)^T, $$
  where $\ell=(\ell_1,\ell_2,\ell_3)$ is a row vector in $\CC^3$. Since $U(x,y) = (x,y,1) \left(\begin{smallmatrix} 2 u_1 & u_2  & u_4 \, \\
u_2 & 2 u_3 & u_5 \, \\
u_4 & u_5 & 2 u_6 \end{smallmatrix}\right)(x,y,1)^T$, the condition
for $U$ to be a square is equivalent to
\begin{equation}
\label{eq:3by3}
 {\rm rank}\begin{pmatrix} 2 u_1 & u_2  & u_4 \, \\
u_2 & 2 u_3 & u_5 \, \\
u_4 & u_5 & 2 u_6 \end{pmatrix}\,\,\, \leq \,\,\, 1 .
\end{equation}

This discussion leads us to the following algebraic reformulation of
Steiner's conic problem:
\begin{equation}\label{ourproblem}
  \begin{array}{l}
    \text{\em Find all solutions $U$ of the equations (\ref{eq:formulation1})} \\
    \text{\em such that the matrix in (\ref{eq:3by3}) has rank $\geq 2$.}
  \end{array}
\end{equation}

Ronga, Tognoli and Vust \cite{RTV}
proved the existence of five real conics whose
3264 conics all have real coefficients. In their argument they do not give an explicit instance, but rather show that in the neighborhood of some particular conic arrangement there must be an instance that has all of the $3264$ conics real. Hence, this raises the following problem:
\begin{equation}\label{ourproblem2}
  \begin{array}{l}
    \text{\em Find an explicit instance $A,B,C,D,E$ such} \\
    \text{\em that the $3264$ solutions $U$ to
    (\ref{ourproblem}) are all real.}
  \end{array}
\end{equation}

Using numerical algebraic geometry we discovered the solution in Figure \ref{fig:real_instance}. We claim that all the $3264$ conics that are tangent to those five conics are real.

\begin{figure*}
  \begin{center}
  \begin{equation*}
    \begin{array}{l}
    \begin{bmatrix}
a_1 & a_2 & a_3 & a_4 & a_5 & a_6 \\
b_1 & b_2 & b_3 & b_4 & b_5 & b_6 \\
c_1 & c_2 & c_3 & c_4 & c_5 & c_6 \\
d_1 & d_2 & d_3 & d_4 & d_5 & d_6 \\
e_1 & e_2 & e_3 & e_4 & e_5 & e_6
\end{bmatrix} \,\,=\,\,
\begin{bmatrix}
  \frac{10124547}{662488724} & \frac{8554609}{755781377} &	\frac{5860508}{2798943247} &	\frac{-251402893}{1016797750} &	\frac{-25443962}{277938473} & 	1\\[0.5em]
  \frac{520811}{1788018449}&	\frac{2183697}{542440933}&	\frac{9030222}{652429049}&	\frac{-12680955}{370629407}&	\frac{-24872323}{105706890} &	1\\[0.5em]
  \frac{6537193}{241535591}&	\frac{-7424602}{363844915}&	\frac{6264373}{1630169777}&	\frac{13097677}{39806827}&	\frac{-29825861}{240478169}& 	1\\[0.5em]
  \frac{13173269}{2284890206}&	\frac{4510030}{483147459}&	\frac{2224435}{588965799}&	\frac{33318719}{219393000}&	\frac{92891037}{755709662} & 1\\[0.5em]
  \frac{8275097}{452566634}&\frac{-19174153}{408565940}&	\frac{5184916}{172253855}&	\frac{-23713234}{87670601}&	\frac{28246737}{81404569} & 1
\end{bmatrix}.
\end{array}
\end{equation*}
\end{center}
\caption{The five conics from Proposition \ref{eins}.\label{fig:real_instance}}
\end{figure*}

\begin{proposition} \label{eins}
There are $3264$ real conics tangent to those
given by the  $5 \times 6$ matrix in Figure \ref{fig:real_instance}.
\end{proposition}
We provide an animation showing all the 3264 real conics of this instance at this URL:
\begin{equation*}
 \text{\url{www.juliahomotopycontinuation.org/3264/}}
\end{equation*}
The construction of our example originates from an arrangement of double lines, which we call the \emph{pentagon construction}. One can see the pentagon in the
middle of Figure \ref{fig:pentagon}. There are points where the red conics seems to intersect a blue line but they are actually points where the red conic touches one branch of a blue hyperbola. See \cite{Sot3264} for further details.

Later we shall discuss the algebro-geometric meaning of the
pentagon construction,  and we present a rigorous computer-assisted proof
 that indeed all of the $3264$ conics tangent to our five conics are real.
 But, first, let us introduce our web browser interface.

\section*{Do It Yourself}

In this section we invite you, the reader, to chose your own instance of five conics.
We offer a convenient way for you to compute the $3264$ complex conics
that are tangent to your chosen conics.
Our web interface for solving
instances of Steiner's problem is found at
\begin{equation}
\label{eq:webbrowser}
 \text{\url{juliahomotopycontinuation.org/diy/}}
 \end{equation}
 Here you can type in your own $30 = 5 \times 6$ coefficients
for the conics in~(\ref{eq:instance}). After specifying five conics,
you press a button and this calls the numerical algebraic geometry software
{\tt HomotopyContinuation.jl}. This is the open source {\tt Julia} package
described in \cites{BT}.
Those playing with the web interface need not
worry about the inner workings.
But, if you are curious,
please read our section titled {\em How Does This~Work?}

 \begin{figure*}
   \begin{center}
 \includegraphics[width = \textwidth]{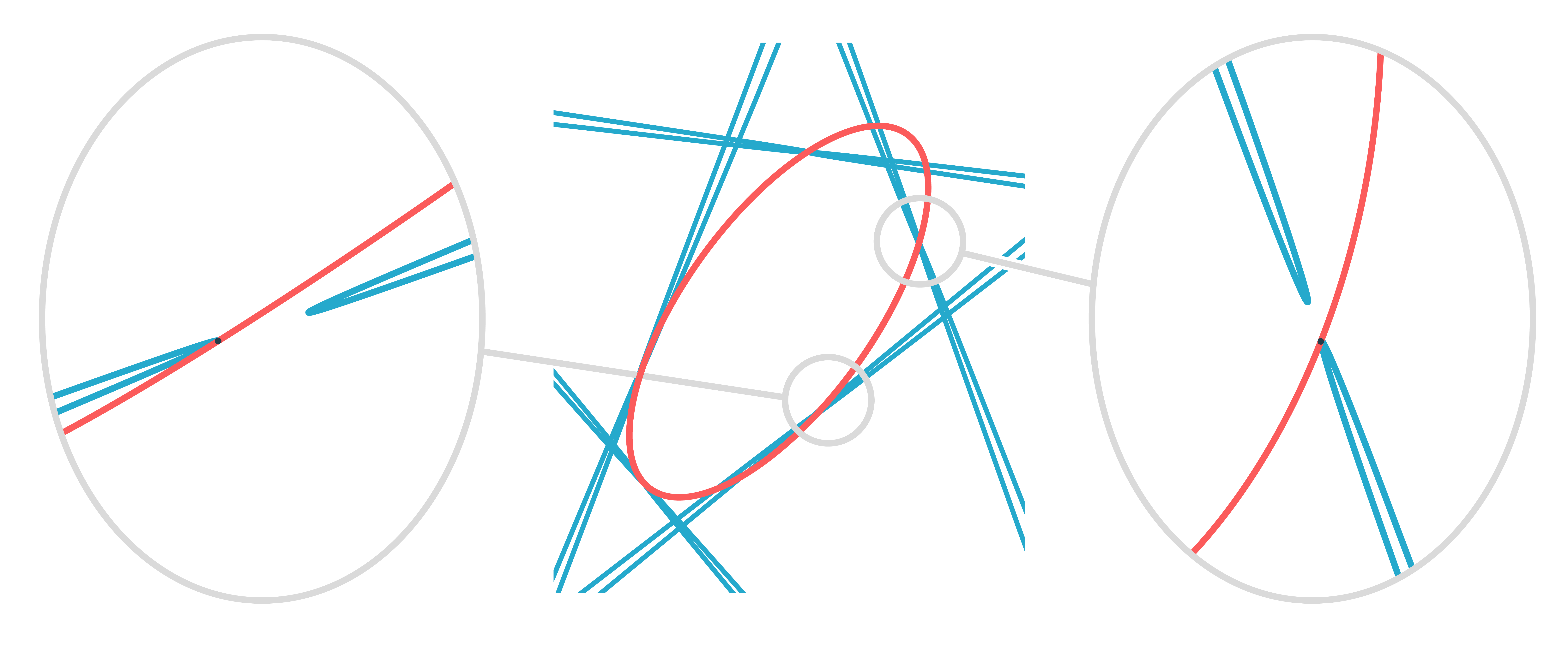}
 \vspace{-0.43in}
 \end{center}
   \caption{The five blue conics in the
   central picture are those in Proposition \ref{eins}.
   Shown in red is one of the $3264$ real conics that are tangent to the
    blue conics. Each blue conic looks like
 a pair of lines, but it is  a thin hyperbola whose branches are close
 to each other. The
 two pictures on the sides show close-ups around two  of the five points of tangency.
  The red conic is tangent to one of the  two branches of the blue hyperbola.
  \label{fig:pentagon}}
 \end{figure*}

\smallskip

Shortly after the user submits their instance,
by entering real coefficients, the web interface reports
whether the instance was generic enough to yield $3264$ distinct complex
solutions. These solutions are computed numerically.
The browser displays
 the number of real solutions, along
with a picture of the instance and a rotating sample of real solutions.
As promised in our title, the computation of all solutions takes only around one second.

\begin{figure}[t]
\begin{center}
\includegraphics[width = 0.42\textwidth]{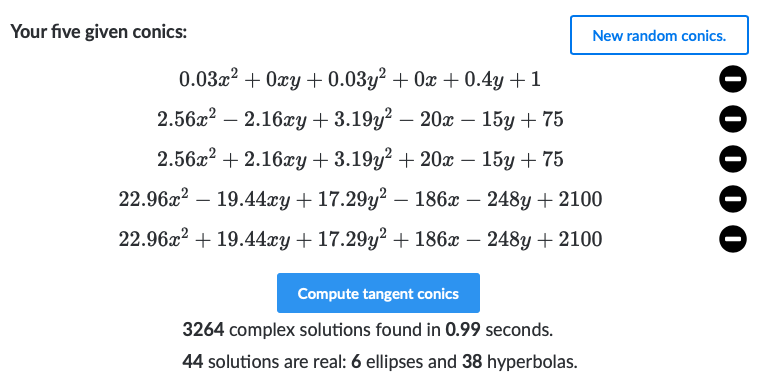}\\[1em]
\includegraphics[width = 0.42\textwidth]{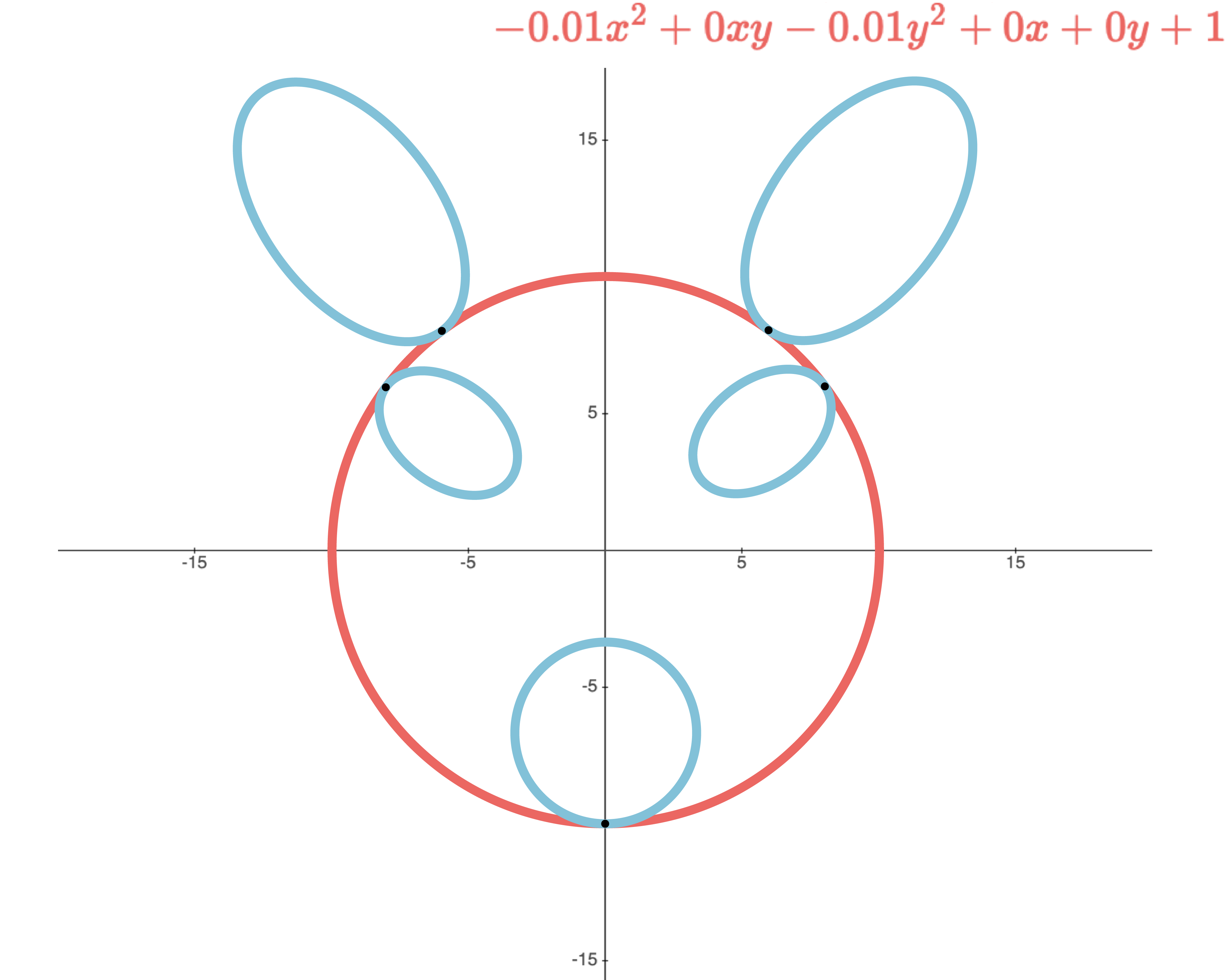}
\end{center}
\vspace{-0.15in}
        \caption{Input and output of the web interface
         (\ref{eq:webbrowser}).
         }
          \label{fig:bunny}
\end{figure}

\begin{remark} \rm
We always assume that the five
given conics are real and generic. This ensures that there are
$3264$ complex solutions, and these conics
are tangent to the given conics at $ 5 \times 3264$ distinct points.
The number of real solutions is even, and
our web interface displays them sequentially.
{\em For every real solution, the points of tangency on the given conics are also real}.
This fact uses the genericity assumption, since two particular real conics
can be tangent at two complex conjugate points. For instance,
the conics defined by $x^2-y^2+1$ and  $ x^2-4y^2+1$ are
tangent at the points $(i:0)$ and $(-i:0)$ where $i = \sqrt{-1}$.
\end{remark}

Figure \ref{fig:bunny} shows what the input and the visual output
of our web interface looks like. The user inputs five conics
and the system shows these in blue. After clicking the
``compute'' button, it responds
with the number of complex and real conics that were found. The $3264$
conics, along with all points of tangency, are available to the
user upon request. The real conics are shown in red,
as seen on the right in  Figure~\ref{fig:bunny}.

When seeing this output, the user might ask a number of questions.
For instance, among the real conics, how
many are ellipses and how many are hyperbolas?
Our web interface answers this question.
The distinction between ellipses and hyperbolas
is characterized by the eigenvalues of the real symmetric matrix
$$ \begin{pmatrix}
2 u_1 & u_2 \\
  u_2 & 2 u_3
  \end{pmatrix} .$$
If the two eigenvalues of this matrix have opposite signs, then the conic is a hyperbola.
If they have the same sign, then the conic is an ellipse. Among the ellipses,
we might ask for the solution which looks most like a circle. Our program does this
by minimizing the expression $\,(u_1-u_3)^2 + u^2_2$. Users with a numerical
analysis background might
 be interested in maximizing
the distance to the
degenerate conics. Equivalently, we ask: among all $3264$ solutions,
which $3 \times 3$ matrix in (\ref{eq:3by3}) has the smallest {\em condition number}?

You can adapt all this
for your favorite geometry problems. As pointed out above, the {\tt Julia} package
{\tt HomotopyContinuation.jl} is available to everyone -- follow the link at \cite{BT}.
This may enable you to solve your own polynomial systems in record time.

\section*{Chow Rings and Pentagons}
We next present the
approach to deriving the number $3264$ that would be taught in an algebraic
geometry class, along the lines of the article \cite{BKT}.
Thereafter we explain the geometric degeneration we used to construct the fully real instance in
Proposition \ref{eins}.

Steiner phrased his problem as that of solving five equations
of degree six on the five-dimensional space $\PP^5_\CC$. The
incorrect count occurred because of the locus of double
conics in~$\PP^5_\CC$. This is a  surface of
extraneous solutions. One fixes the problem by
replacing $\PP^5_\CC$ with another five-dimensional
manifold, namely the {\em space of complete conics}.
This space is the blow-up of $\PP^5_\CC$ at the locus of double lines.
It is a compactification of the space of nonsingular conics that has
desirable geometric properties.
A detailed description of this construction,
suitable for a first course in algebraic geometry,
can be found in \cite[\S 5.1]{BKT}.

In order to answer enumerative geometry questions about
the space of complete conics, one considers its
Chow ring, as explained in \cite[\S 5.2]{BKT}.
Elements in the Chow ring of the space of complete conics correspond to subvarieties of this space. More precisely, to \emph{classes} of subvarieties. Two subvarieties belong to the same class if and only if they are \emph{rationally equivalent}. Rational equivalence is a technical concept. We refer interested readers
to the textbook by Eisenbud and Harris \cite{EH}.
The Chow ring for the space of complete conics
is worked out in \cite[\S 8.2.4]{EH}. Nevertheless, the idea behind studying Chow rings is crystal clear: taking intersections of varieties is translated to multiplication in the Chow ring.
In the remainder of this section we will see this in action.

The Chow ring of the space of complete conics contains two special
classes $P$ and $L$. The class $P$ encodes the
conics passing through a fixed point, while the class $L$ encodes the
conics tangent to a fixed line.
The following relations hold in the Chow ring:
$$ P^5 = L^5 = 1  ,\, P^4 L = P L^4 = 2 ,\,P^3 L^2 = P^2 L^3 = 4. $$
These relations are derived in \cite[\S 4.4--5.3]{BKT}. For instance, the first equation means that, if we take five general conics passing through a fixed point, then the intersection contains one point (namely the point we fixed in the first place). See
\cite[Table 3]{BKT} for the geometric meaning of the other equations.

We write $C$ for the class of conics that are tangent to a given conic.
In the Chow ring, we have
$$ C \,\, = \,\, 2P + 2L . $$
This identity is derived in \cite[equation (8)]{BKT}.
Our preferred proof  is to inspect the first three terms
in the expression (\ref{eq:tact}) for the tact invariant $\mathcal{T}(A,U)$:
$$ \mathcal{T} =  16 \cdot u_6^2 (4u_3u_6-u_5^2)^2 \cdot a_1^4 a_2^3
\quad {\rm mod}  \langle a_2,a_3^3, a_4,a_5,a_6 \rangle . $$
This has the following intuitive interpretation. We assume that the given fixed conic $A$ satisfies
\begin{equation}
\label{eq:gg} |a_1| \gg |a_3| \gg \max \{ |a_2|, |a_4|, |a_5|, |a_6| \}.
\end{equation}
Thus the conic $A$ is close to $x^2 - \epsilon y^2$, where
$\epsilon$ is a small quantity. The process of letting $\epsilon$ tend to zero
is understood as a degeneration in the sense of algebraic geometry.
With this, the condition for $U$ to be tangent to $A$
degenerates to $\,u_6^2 \cdot (4u_3u_6-u_5^2)^2 = 0$.

The first factor $u_6$ represents all
conics that pass through the point $(0,0)$. The second factor $4u_3u_6-u_5^2$
represents all conics tangent to the line $\{x=0\}$. The Chow ring classes
of these factors are $P$ and $L$. Each of these arises with multiplicity~$2$,
as seen from the exponents.
The desired intersection number is now obtained from the Binomial Theorem:
\begin{small}
  \begin{equation*}
    \begin{array}{rl}
      C^5 \mkern-18mu
    &= 32 (L+P)^5 \\
      &= 32 ( L^5 + 5 L^4 P + 10 L^3 P^2 + 10 L^2 P^3 + 5 LP^4 + P^5) \\
      &= 32 (1 + 5 \cdot 2 + 10 \cdot 4 + 10 \cdot 4 + 5 \cdot 2 + 1) \\
      &= 32 \cdot 102  = 3264 \,.
    \end{array}
  \end{equation*}
\end{small}
The final step in turning this into a rigorous proof of Chasles' result is carried out
in \cite[\S 7]{BKT}.

The degeneration idea in (\ref{eq:gg}) can be used to construct
real instances of Steiner's problem
whose $3264$ solutions are all real.
Fulton first observed this and communicated it to Sottile, who then wrote down Fulton's proof in detail \cite{Sot, Sot3264}. Ronga, Tognoli and Vust \cite{RTV} independently gave a  proof.
Apparently, they did not know
about Fulton's ideas.

Fix a convex pentagon in $\RR^2$ and one special point
somewhere in the relative interior of each edge. Consider all conics~$C$
such that, for each edge of the pentagon, $C$ either passes through the
special point or is tangent to the line spanned by the edge. By the count above,
there are $(L+P)^5 = 102$ such conics~$C$.
 If the pentagon is chosen sufficiently asymetric, then the $102$ conics are all real.
We now replace each pointed edge by a nearby hyperbola, satisfying~(\ref{eq:gg}). For instance, if the edge has equation $x=0$ and $(0,0)$ is its special point
then we take the hyperbola  $x^2 - \epsilon y^2 + \delta$, 
where $\epsilon > \delta > 0$ are very small. After making appropriate choices
of these parameters along all edges of the pentagon, each of the
$102$ conics splits into $32$ conics, each tangent to the five hyperbolas. Here 'splits' means, if the process is reversed, then the $ 32$ different conics  collapse into one solution of multiplicity 32. By construction,
all $3264$ conics are real.

The argument shows that there \emph{exists} an instance in the neighborhood of the pentagon
whose $3264$ conics are all real, but it does not say \emph{how close}
they should be.
Serious hands-on experimentation was necessary for finding the instance
 in Proposition~\ref{eins}.

\smallskip

We next present an {\bf alternative formulation} of Steiner's conic problem.
The idea is to remember the five points of tangency on each solution conic.
The five sextics in~(\ref{eq:formulation1})
did not involve these points. They were obtained directly from the
tact invariant. The next system of equations avoids the use of the tact invariant.
It uses five copies of the equations (\ref{eq:conicA})--(\ref{eq:jacobian}), each with a different point of tangency
 $(x_i,y_i)$, for $i=1,2,3,4,5$.
The ten equations from (\ref{eq:conicA}) and (\ref{eq:conicU}) are quadrics.
The five equations from~(\ref{eq:jacobian}) are cubics. Altogether, we get the following system of $15$
equations that we display as a $5\times 3$ matrix $F_{(A,B,C,D,E)}$:
\begin{small}
\begin{equation}\label{alternative_formulation}
  \setlength\arraycolsep{2pt}
 \begin{bmatrix}
\, A(x_1,y_1)&
U(x_1,y_1) &
(\frac{\partial A}{ \partial x}   \frac{\partial U}{\partial y} \,-\,
\frac{\partial A}{ \partial y}   \frac{\partial U}{\partial  x})(x_1,y_1) \, \\[0.5em]
\, B(x_2,y_2)&
U(x_2,y_2) &
(\frac{\partial B}{ \partial x}   \frac{\partial U}{\partial y} \,-\,
\frac{\partial B}{ \partial y}   \frac{\partial U}{\partial  x})(x_2,y_2) \, \\[0.5em]
\, C(x_3,y_3)&
U(x_3,y_3) &
(\frac{\partial C}{ \partial x}   \frac{\partial U}{\partial y} \,-\,
\frac{\partial C}{ \partial y}   \frac{\partial U}{\partial  x})(x_3,y_3) \, \\[0.5em]
\, D(x_4,y_4)&
U(x_4,y_4) &
(\frac{\partial D}{ \partial x}   \frac{\partial U}{\partial y} \,-\,
\frac{\partial D}{ \partial y}   \frac{\partial U}{\partial  x})(x_4,y_4) \, \\[0.5em]
\, E(x_5,y_5)&
U(x_5,y_5) &
(\frac{\partial E}{ \partial x}   \frac{\partial U}{\partial y} \,-\,
\frac{\partial E}{ \partial y}   \frac{\partial U}{\partial  x})(x_5,y_5) \,
\end{bmatrix}.\end{equation}
\end{small}

Each matrix entry is a polynomial in the $15$ variables
$u_1,\ldots , u_5, x_1,y_1\ldots,x_5,y_5$.
The parameters of this system are the coefficients of the conics $A,B,C,D,E$. The system
of five equations  seen in
(\ref{eq:formulation1}) is obtained by eliminating  the $10$ variables
$x_1,y_1,x_2,y_2, x_3,y_3,
 x_4,y_4,  x_5,y_5 $ from the new system
  $ F_{(A,B,C,D,E)}(x) $ introduced in (\ref{alternative_formulation}).

\smallskip

At first glance, it looks like the new formulation  (\ref{alternative_formulation})
is worse than the one in (\ref{eq:formulation1}). Indeed,
the number of variables has increased from $6$ to $15$,
and the B\'ezout number has increased from $\,6^5 = 7776\,$
to $\, 2^{10} 3^5 =   248832 $. However, the new formulation is better suited for
 the numerical solver that powers our website. We explain this in the last section.

\section*{Approximation and Certification}
Steiner's conic problem amounts to solving a system of polynomial equations.
Two formulations were given in (\ref{eq:formulation1}) and (\ref{alternative_formulation}). But what does 'solving' actually mean? One answer is suggested in the
 textbook  by Cox, Little and O'Shea \cite{CLO}:
 Solving means computing a {\em Gr\"obner basis} $\mathcal{G}$. Indeed,
crucial invariants, like the dimension and degree of the solution variety,
are encoded in $\mathcal{G}$. The number of real solutions
is found by applying techniques like deriving {\em Sturm sequences} from the polynomials in
$\mathcal{G}$.  Yet, Gr\"obner bases can take a very long time to compute. We found them impractical for Steiner's problem.

Computing $3264$ conics in a second requires numerical methods.
Our encoding of the solutions
are not Gr\"obner bases but \emph{numerical approximations}.
How does one make this rigorous? This question can be phrased as follows.
Suppose $u_1,\ldots,u_6$ are the true coordinates of a solution
and $u_1+\Delta u_1,\ldots,u_6+\Delta u_6$ are approximations of those
complex numbers.
How small must the entries of  $\Delta u_1, \ldots, \Delta u_6$ be before it
is justified to call them approximations? This question is elegantly circumvented by using Smale's
definition of \emph{approximate zero} \cite[Definition 1 in \S 8]{BCSS}.

In short, an approximate zero of a system $F(x)$ of~$n$ polynomials in $n$ variables is any point $z\in \mathbb{C}^n$ such that Newton's method when applied to $z$ converges quadratically fast towards a zero of $F$. Here is the precise definition.

\begin{definition}[Approximate zero]\label{def_approx_zero} \rm
Let
$J_F(x)$ be the $n \times n$ Jacobian matrix of $F(x)$.  A point $z \in \mathbb{C}^n$ is an
{\em approximate zero} of $F\,$ if there exists
a zero $\zeta\in \mathbb{C}^n$ of~$F$ such that the sequence of  Newton iterates
$$z_{k+1} = N_F(z_k), \text{ where }
 N_F(x) =  x - J_F(x)^{-1}F(x)
$$
 starting at $z_0=z$
satisfies or all $k=1,2,3,\ldots$ that
$$ \Vert z_{k+1} - \zeta\Vert \,\leq \, \frac{1}{2}\Vert z_k-\zeta\Vert^2.$$
If this holds, then we call $\zeta$ the {\em associated zero} of~$z$.
Here~$\Vert x\Vert := (\sum_{i=1}^n x_i\overline{x_i})^\frac{1}{2}$
is the standard norm in~$\mathbb{C}^n$, and the zero $\zeta$ is assumed to be
nonsingular, i.e.~${\rm det}(J_F(\zeta)) \not=0 $.
\end{definition}

The reader should think of approximate zeros as a \emph{data structure}
for representing solutions to polynomial systems, just like a Gr\"obner basis is
 a data structure. Different types of representations of data provide different levels of accessibility to  the desired information. For instance, approximate zeros are not well suited for computing
 algebraic features of an~ideal. But they are a powerful tool for
 answering geometric questions in a fast and reliable manner.

Suppose that $z$ is a point in $\CC^n$ whose real and imaginary part are rational numbers.
How can we tell whether $z$ is an approximate zero of $F$? This is not clear from the definition.

It is possible to certify that $z$ is an approximate zero
without dealing with the infinitely many Newton iterates.
We next explain how this works. This involves Smale's $\gamma$-number and Smale's $\alpha$-number:
\begin{align*}
\gamma(F,z) &\,\,=\,\, \sup_{k\geq 2}\big\Vert \frac{1}{k!}\,J_F(z)^{-1} D^kF(z)\big\Vert^\frac{1}{k-1},\\
\alpha(F,z)  &\,\,=\,\,  \Vert J_F(z)^{-1}F(z)\Vert  \cdot \gamma(F,z).
\end{align*}
Here $D^kF(z)$ denotes the tensor of order-$k$ derivatives at the point $z$,
the tensor $J_F(z)^{-1} D^kF(z)$ is understood as a multilinear map
$A:(\mathbb{C}^n)^k\to \mathbb{C}^n$, and the norm
of this map is $\Vert A\Vert := \max_{\Vert v \Vert = 1} \Vert A(v,\ldots,v)\Vert$.

Shub and Smale \cite{SS} derived an upper bound for $\gamma(F,z)$ which can be computed exactly. Based on the next theorem \cite[Theorem 4 in Chapter 8]{BCSS}, one can thus decide algorithmically
if $z$ is an approximate zero, \emph{using only data of the point~$z$ itself}.

\begin{theorem}[Smale's $\alpha$-theorem]\label{alpha_theorem}
If $\alpha(F,z)<0.03$, then $z$ is an approximate zero of $F(x)$. Furthermore, if $y\in\mathbb{C}^n$ is any point with $\Vert y-z\Vert < (20\,\gamma(F,z))^{-1}$, then $y$ is also an approximate zero of $F$ with the same
associated zero $\zeta$ as $z$.
\end{theorem}
Actually, Smale's $\alpha$-theorem is more general in the sense that $\alpha_0=0.03$ and $t_0=20$ can be replaced by any two positive numbers $\alpha_0$ and $ t_0$ that satisfy a certain list of inequalities.

Hauenstein and Sottile \cite{HS} use  Theorem \ref{alpha_theorem}
 in~an algorithm, called {\tt alphaCertified}, that decides if a point $z\in \mathbb{C}^n$ is an approximate zero and if two approximate zeros have distinct associated solutions. An implementation  is publicly available. Furthermore, if the polymomial system $F$ has only real coefficients, then {\tt alphaCertified} can decide if an associated zero is real. The idea behind this is as follows: Let $z\in \mathbb{C}^n$ be an approximate zero of~$F$ with associated zero~$\zeta$.
If the coefficients of $F$ are all real, then the Newton operator $N_F(x)$ from Definition \ref{def_approx_zero} satisfies $N_F(\overline{x}) =\overline{ N_F(x)}$.
 Hence $\overline{z}$ is  an approximate zero of $F$ with associated zero $\overline{\zeta}$. If~$\Vert z-\overline{z}\Vert < (20\,\gamma(F,z))^{-1}$, then, by Theorem \ref{alpha_theorem}, the associated zeros of $z$ and $\overline{z}$ are equal.
 This means $\zeta = \overline{\zeta}$.

A fundamental insight is that Theorem \ref{alpha_theorem} allows us to certify
candidates for approximate zeros regardless of how they were obtained. Typically,
  candidates are found by inexact computations using floating point arithmetic.
    We do not need to know what happens in that computation,
  because we can certify the result  \emph{a posteriori}.
  Certification constitutes a rigorous proof of a mathematical result.
  Let us see this in action.

  \begin{figure}[t]
  \begin{center}
  \includegraphics[height = 3.7cm]{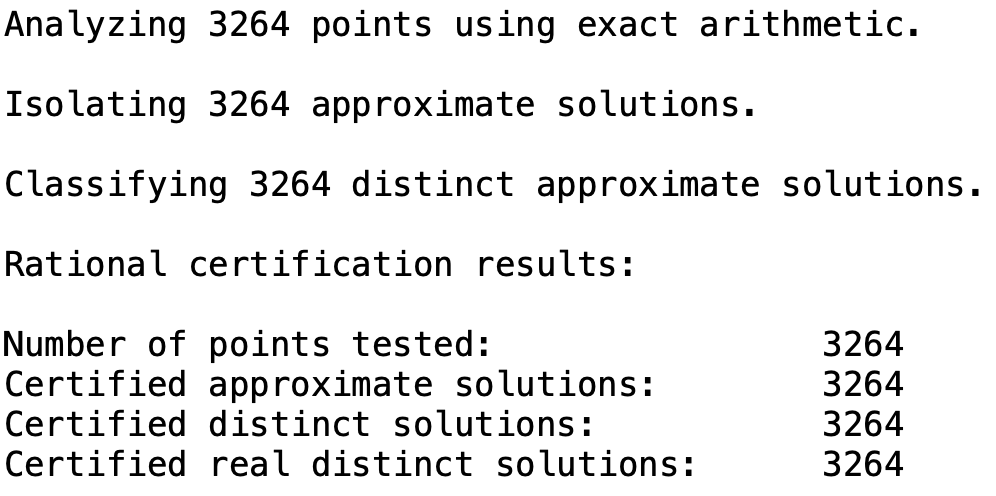}\vspace{-0.2in}
      \end{center}
  \caption{A proof for Proposition \ref{eins} given by the software \texttt{alphaCertified}.\label{output_alpha_certified}}
  \end{figure}
  \begin{proof}[Proof of Proposition \ref{eins}]
Fix five non-degenerate conics with rational coefficients listed after equation (\ref{ourproblem2}). We apply {\tt HomotopyContinuation.jl} \cite{BT}
    to compute $3264$ solutions in a second in 64-bit floating point arithmetic.
    The output is inexact. Each coefficient $u_i$ of each true solution $U$
    is a complex number
    that is algebraic of degree $3264$ over $\QQ$.
    The floating point numbers that represent these coefficients are
    rational numbers and we treat them as elements of $\QQ$.

Our proof starts with the resulting list of $3264 $ vectors $x \in \QQ^{15}$
corresponding to the $15 $ variables of \eqref{alternative_formulation}.
The computation was mentioned to make the exposition more friendly.
It is  {\em not} part of the proof.

We are now given $3264$ candidates for approximate zeros
 of the polynomial system in (\ref{alternative_formulation}).
These candidates have rational coordinates. We use them
as input to the software {\tt alphaCertified} from  \cite{HS}.
That software performs \emph{exact computations in rational arithmetic}.
Its output shows that the $3264$ vectors $x$ are approximate zeros, that their associated zeros $\zeta$ are distinct, and that they all have real coordinates.
This is shown in Figure \ref{output_alpha_certified}. The output data of {\tt alphaCertified} is available for download through the arXiv version of this article.
  \end{proof}

This was a rigorous proof of Proposition \ref{eins}, just as trustworthy as a computer-assisted proof by symbolic computation
(e.g.~Gr\"obner bases and Sturm sequences) might have been. Readers
who are experts in algebra should
not get distracted by the appearance of floating point arithmetic: it is \emph{not} part of the proof! Floating point numbers are only a tool for obtaining the $3264$ candidates. The actual proof is carried out by exact symbolic computations.

\section*{How Does This Work?}
In this section we discuss the methodology and software that powers  the web interface (\ref{eq:webbrowser}).

We use the software {\tt HomotopyContinuation.jl} that was developed by two of us \cite{BT}.
This is a {\tt Julia} \cite{BEKV} implementation of a computational paradigm called \emph{homotopy continuation}. The reasons we chose {\tt Julia} as the programming language are threefold: the first is that {\tt Julia} is open source and free for anyone to use. The second is that {\tt Julia} can be as fast as well written {\tt C}. For instance, we use {\tt Julia}'s JIT compiler for fast evaluation of polynomials. Finally, the third reason is that, despite its high performance, {\tt Julia} still provides an easy high-level syntax. This makes our software accessible for users from many backgrounds.

Homotopy continuation works as follows: We wish to find a zero in $\CC^n$ of
a system $F(x)$ of $n$ polynomials in $n$ variables. Let
 $G(x)$ be another such system with a known zero $G(\zeta) = 0$. We connect $F$ and $G$ in the space of polynomial systems
 by a path $t\mapsto H(x,t)$ with $H(x,0) = G(x)$ and $H(x,1)=F(x)$.

 The aim is to approximately follow the {\em solution path} $x(t)$ defined by $H(x(t),t)=0$. For this, the path is discretized into
 steps $\,t_0 = 0 < t_1 < \cdots < t_k = 1$. If the discretization is fine enough
 then $\zeta$ is also an approximate zero of $H(x,{t_1})$.
Hence, by Definition~\ref{def_approx_zero}, applying the Newton operator
 $N_{H(x,t_1)}(x)$ to $\zeta$, we get a sequence $\zeta_0,\zeta_1,\zeta_2,... $ of points that converges towards a zero $\xi$ of $H(x,t_1)$. If $t_2-t_1$ and  $\Vert \zeta_i-\xi\Vert$ are small enough, for some $i\geq 0$, then the iterate $\zeta_i$ is an approximate zero of $H(x,t_2)$.

  We may repeat the procedure for $H(x,t_2)$ and starting with $\zeta_i$. Inductively, we find an approximate zero of $H(x,t_j)$ for all~$j$. In the end, we obtain
  an approximate zero for the system $F(x)=$ $ H(x,1)$. Most implementations of homotopy continuation, including {\tt Bertini} \cite{BHSW06} and {\tt HomotopyContinuation.jl} \cite{BT}, use heuristics for setting
  both the step sizes $t_{j+1}-t_j$ and the number of Newton iterations.

Our homotopy for Steiner's conic problem computes zeros of the system~$F_{(A,B,C,D,E)}(x)$ from (\ref{alternative_formulation}).
We prefer formulation (\ref{alternative_formulation}) over (\ref{eq:formulation1}) because the
equations in the former formulation have
lower degrees and fewer terms.
It is known that high degrees and many terms introduce numerical instability in the evaluation of polynomials.
We use the homotopy
\begin{equation}\label{our_homotopy}
  \mkern-12mu H(x,t) = F_{t\cdot(A,B,C,D,E)+(1-t)\cdot(A',B',C',D',E')}(x).
\end{equation}
The conic $tA+(1-t)A'$ is defined by the coefficients $ta_i+(1-t)a_i'$, where $a_i$ and $a_i'$ are the coefficients of $A$ and $A'$.
This is called a \emph{parameter homotopy} in the literature.
Geometrically, (\ref{our_homotopy})
 is a straight line in the space of quintuples of conics.
An alternative would have been
  \begin{equation}\label{not_our_homotopy}
  \begin{small} \mkern-6mu
    {\tilde H}(x,t) =
    tF_{(A,B,C,D,E)} + (1-t)F_{(A',B',C',D',E')}.
\end{small}
  \end{equation}
The advantage of  (\ref{our_homotopy}) over
(\ref{not_our_homotopy}) is that the path
stays within the space of \emph{structured systems}
$$\{F_{(A,B,C,D,E)}(x) \mid (A,B,C,D,E) \text{ are conics}\}.$$
The structure of the equations is preserved. The system in (\ref{our_homotopy}) has $3264$ solutions for almost all $t$,
whereas (\ref{not_our_homotopy}) has $7776$ solutions for random $t$.
In the language of algebraic geometry,
we prefer the {\em flat family} (\ref{our_homotopy}) over
(\ref{not_our_homotopy}), which is not flat.

The last missing piece in our {\em Steiner homotopy} is a start system. That is, we need five
explicit conics together with all~3264 solutions. For this, we construct a generic instance $(A,B,C,D,E)$ by randomly sampling complex coefficients for the conics. Then, we compute the 3264 solutions using standard homotopy continuation techniques \cite{SW}. Those 3264 solutions are saved and used for further computations. This initial computation is significantly more expensive than tracking 3264 solutions along the homotopy (\ref{our_homotopy}) but it only has to be done once.

In closing, we emphasize the important role played by enumerative geometry for solving polynomial systems.
It gives a criterion for deciding if the initial numerical computation found the correct number of solutions. This is why
numbers like $3264$ are so important, and why a numerical analyst might~care about
Chow rings and the pentagon construction.
   We conclude that enumerative algebraic geometry and numerical algebraic geometry
  complement each other.

\bigskip \bigskip

\noindent
\footnotesize {\bf Authors' addresses:}

\smallskip

\noindent Paul Breiding,  MPI-MiS Leipzig
\hfill {\tt breiding@mis.mpg.de}

\noindent Bernd Sturmfels,
 \\  \phantom{dodo} MPI-MiS Leipzig and
UC  Berkeley \hfill  {\tt bernd@mis.mpg.de}

\noindent Sascha Timme, TU Berlin
\hfill {\tt timme@math.tu-berlin.de}

\begin{thebibliography}{10}

\begin{small}
\setlength{\itemsep}{-0.4mm}

\bibitem[BKT08]{BKT}
A.~Bashelor, A.~Ksir and W.~Traves:
{\em Enumerative algebraic geometry of conics},
American Mathematical Monthly {\bf 115} (2008) 701--728.
MR2456094

\bibitem[BHSW06]{BHSW06}
D.~Bates, J.~Hauenstein, A.~Sommese and C.~Wampler:
{\em Bertini: Software for Numerical Algebraic Geometry.}
Available at \url{bertini.nd.edu},

\bibitem[BEKV17]{BEKV}
J.~Bezanson, A.~Edelman, S.~Karpinski and V.~Shah:
\emph{Julia: A fresh approach to numerical computing},
SIAM Review {\bf 59} (2017) 65--98.
MR3605826

\bibitem[BCSS98]{BCSS}
L.~Blum, F.~Cucker, M.~Shub and S.Smale:
{\em Complexity and Real Computation},
Springer, 1998.
MR1479636


 \bibitem[BT18]{BT}  P.~Breiding and S.~Timme:
 {\em Homotopy{C}ontinuation.jl - a package for solving systems of
 polynomial equations in {J}ulia},
 Mathematical Software -- ICMS 2018. Lecture Notes in Computer Science. Software available at
{\tt www.juliahomotopycontinuation.org}.



\bibitem[CLO15]{CLO}
D.~Cox,  J.~Little and D.~O'Shea:
{\em Ideals, Varieties, and Algorithms}:
 An Introduction to Computational Algebraic Geometry and
 Commutative Algebra, 4th ed., Undergraduate Texts in Mathematics, Springer,~2015.
 MR3330490


\bibitem[EH16]{EH}
D.~Eisenbud and J.~Harris:
{\em 3264 And All That -- a Second Course in Algebraic Geometry},
 Cambridge University Press, 2016.
 MR3617981

 \bibitem[GZ79]{GZ}
 C.B.~Garcia and W.L.~Zangwill:
 {\em Finding all solutions to polynomial systems and other systems of equations},
 Mathematical Programming {\bf 16} (1979) 159-176. MR0527572

\bibitem[HS12]{HS}
J.~Hauenstein and F.~Sottile:
{\em alphaCertified: Certifying solutions to polynomial systems},
ACM Transactions on Mathematical Software {\bf 48} No.~4 (2012).
MR2972672

\bibitem[RTV97]{RTV}
F.~Ronga, A.~Tognoli and T.~Vust:
{\em The number of conics tangent to five given conics: the real case},
Rev.~Mat.~Univ.~Complut.~Madrid {\bf 10} (1997) 391--421.
MR1605670

\bibitem[SS93]{SS}
M.~Shub and S.~Smale:
{\em Complexity of B\'ezout's theorem I. Geometric aspects},
Journal of the American Mathematical Society {\bf 6} (1993) 459--501.
MR1175980

\bibitem[Sot95]{Sot}
F.~Sottile: {\em Enumerative geometry for real varieties},
Algebraic geometry -- Santa Cruz 1995, 435--447,
Proc.~Sympos.~Pure Math., 62, Part 1, Amer. Math. Soc., Providence, RI, 1997.
MR1492531

\bibitem[Sot]{Sot3264}
F.~Sottile: {\em 3264 real conics}, \url{www.math.tamu.edu/~sottile/research/stories/3264/}.

\bibitem[SW05]{SW}
A.~Sommese and C.~Wampler: {\em The Numerical Solution of Systems of Polynomials Arising in Engineering and Science},
World Scientific, 2005.
MR2160078

\end{small}

\end{thebibliography}
\end{document}